\documentclass{amsart}
      \usepackage{amsmath,amssymb,amsfonts,enumerate,amsthm,graphicx}

      \newcommand{\Spec}{\mbox{Spec}\,}

      \newcommand{\depth}{\mbox{depth}\,}
      
      \renewcommand{\dim}{\mbox{dim}\,}
      \newcommand{\lk}{\mbox{lk}\,}

      \newcommand{\V}{\mbox{V}}

      \newcommand{\fm}{\mathfrak{m}}
      \newcommand{\fp}{\mathfrak{p}}

      \newtheorem{thm}{Theorem}[section]
      \newtheorem{cor}[thm]{Corollary}
      \newtheorem{lem}[thm]{Lemma}
      \newtheorem{prop}[thm]{Proposition}
      \newtheorem{defn}[thm]{Definition}

      \newtheorem{rem}[thm]{Remark}
      \newtheorem{ques}[thm]{Question}

      \numberwithin{equation}{section}
      
\begin{document}
      \bibliographystyle{amsplain}
      \title{Sequentially $S_r$ simplicial complexes and sequentially
      $S_2$ graphs}
      \author{Hassan Haghighi}

      \address{Hassan Haghighi\\Department of Mathematics, K. N. Toosi
      University of Technology, Tehran, Iran.}

      \author{Naoki Terai}
      \address{Naoki Terai\\Department of Mathematics, Faculty of culture and Education,
      SAGA University, SAGA 840-8502, Japan}

      \author{Siamak Yassemi}
      \address{Siamak Yassemi\\School of Mathematics, Statistics \&
      Computer Science, University of
      Tehran, Tehran Iran.}
      \author{Rahim Zaare-Nahandi}
      \address{Rahim Zaare-Nahandi\\School of Mathematics, Statistics \&
      Computer Science, University of Tehran, Tehran, Iran.}
      \thanks{H. Haghighi was supported in part by a grant from K. N. Toosi University of Technology}

      \thanks{Emails: haghighi@kntu.ac.ir, terai@cc.saga-u.ac.jp, yassemi@ipm.ir, rahimzn@ut.ac.ir}



      \begin{abstract}

      \noindent We introduce sequentially $S_r$ modules over a commutative graded ring and
      sequentially $S_r$ simplicial complexes. This generalizes two properties
      for modules and simplicial complexes: being sequentially
      Cohen-Macaulay, and satisfying Serre's condition $S_r$. In
      analogy with the sequentially Cohen-Macaulay property,
      we show that a simplicial complex is sequentially $S_r$ if and
      only if its pure $i$-skeleton is $S_r$ for all $i$.
      For $r=2$, we provide a more relaxed characterization. As an
      algebraic criterion, we prove that a
      simplicial complex is sequentially $S_r$ if and only if the
      minimal free resolution of the ideal of its Alexander dual is
      componentwise linear in the first $r$ steps. We apply these
      results for a graph, i.e., for the simplicial complex of
      the independent sets of vertices of a graph. We characterize sequentially $S_r$ cycles showing
      that the only sequentially $S_2$ cycles are odd cycles and, for $r\ge 3$, no cycle is
      sequentially $S_r$ with the exception of cycles of length $3$ and $5$.
      We extend certain known results
      on sequentially Cohen-Macaulay graphs to the case of sequentially $S_r$ graphs.
      We prove that a bipartite graph is vertex decomposable if and only
      if it is sequentially $S_2$.  We provide some
      more results on
      certain graphs which in particular implies that any graph with no chordless even
      cycle is sequentially $S_2$. Finally, we propose some
      questions.

      \end{abstract}

      \maketitle

      \section{Introduction}

      Let $R=k[x_1,\ldots, x_n]$ be the polynomial ring over a field $k$.
      For finitely generated graded $R$-modules, Stanley has defined
      the sequentially Cohen-Macaulay property \cite[Chapter III, Definition 2.9]{St-95}
      and has studied the corresponding simplicial complexes. Here we consider
      sequentially $S_r$ graded modules, i.e., finitely generated graded
      $R$-modules which satisfy Serre's $S_r$ condition sequentially. Then
      we study the corresponding simplicial complexes, sequentially $S_r$
      simplicial complexes. Duval has shown that a simplicial complex
      is sequentially Cohen-Macaulay if and only if its pure $i$-skeleton
      is Cohen-Macaulay for all $i$ \cite[Theorem 3.3]{Du-96}. We
      prove the analogue result for
      sequentially $S_r$ simplicial complexes (see Theorem \ref{thm:pure skeletons}).
      For $r=2$, we show that a simplicial complex is sequentially $S_2$
      if and only if its pure $i$-skeletons are connected for all $i\ge 1$
      and the link of every singleton is sequentially $S_2$ (see Theorem
      \ref{thm:localcriteria}).
      A major result of Eagon and Reiner states that a simplicial
      complex is Cohen-Macaulay if and only if the Stanley-Reisner ideal
      of its Alexander dual has a
      linear resolution \cite[Theorem 3]{Ea-Re-98}. Later, Herzog and Hibi
      generalized this result by proving that a simplicial complex is
      sequentially Cohen-Macaulay if and only if the minimal free
      resolution of the Stanley-Reisner ideal of its Alexander dual is
      componentwise linear \cite[Theorem 2.9]{He-Hi-99}. The result of
      Eagon and Reiner has been generalized in another direction by Yanagawa
      (with N. Terai) showing
      that a simplicial complex is $S_r$ if and only if the
      minimal free resolution of its Alexander dual is linear in the first
      $r$ steps \cite[Corollary 3.7]{Ya-2000}. We adopt the above two
      results to show that a simplicial complex is sequentially $S_r$
      if and only if the minimal free resolution of the
      Stanley-Reisner ideal of its Alexander dual is componentwise linear
      in the first $r$ steps (see Corollary \ref{Cor:Square_Component}).

      As the first application of our results, we characterize sequentially $S_r$ cycles.
      It is known that the only cycles which are sequentially
      Cohen-Macaulay are $C_3$ and $C_5$
      \cite[Proposition 4.1]{Fr-Va-07} and the only cycles which are  $S_2$,
       are $C_3$, $C_5$ and $C_7$
      \cite[Proposition 1.6]{H-Y-Z-09}.
      We extend these results by showing that $C_n$ is sequentially $S_2$ if and only
      if $n$ is odd  and the only sequentially $S_3$ cycles are $C_3$ and
      $C_5$, i.e., the Cohen-Macaulay cycles (see Theorem \ref{cycles} and Proposition
      \ref{S3cycles}).

      Van Tuyl and Villarreal \cite{V-V-08} have studied
      sequentially Cohen-Macaulay graphs.
      We extend some of their results
      and generalize a result of Francisco and H\`{a} \cite[Theorem
      4.1]{Fr-Ha-08} on graphs with whiskers (see Corollary
      \ref{whisker}).

      Van Tuyl \cite[Theorem  2.10]{V-09} has recently proved that a bipartite graph is
      vertex decomposable if and only if it is sequentially Cohen-Macaulay. We prove
      that a bipartite graph is vertex decomposable if and only if it is
      sequentially $S_2$ (see Theorem \ref{chordlessgeneral} and Corollary \ref{chordlessclass} ).
      This result also generalizes
      the authors'
      result which states that a bipartite graph is Cohen-Macaulay if and
      only if it is $S_2$ \cite[Theorem 1.3]{H-Y-Z-09}.

      Woodroofe \cite[Theorem 1.1]{Wo-08} has proved that a graph with no
      chordless
      cycles other that cycles of length $3$ and $5$ is sequentially
      Cohen-Macaulay. We extend this result for $S_2$ property (see Theorem
      \ref{chordlesswisker}). This in
      particular implies that any graph with no chordless even cycle is
      sequentially $S_2$ (see Corollary \ref{chordless}).

      At the end of this paper we propose two questions on
      sequentially $S_r$ property of join of two simplicial
      complexes and topological invariance of $S_r$, respectively.\\

      The motivation behind our work is the general philosophy that
      Serre's $S_r$ condition plays an important role, not only in
      algebraic geometry and commutative algebra, but also in
      algebraic combinatorics (e.g. see \cite{Sch-82},
      \cite{Ya-2000}, \cite{Ter-07}).

      \section{Criteria for sequentially $S_r$ simplicial complexes}

      In this section we give some basic definitions and
      criteria for sequentially $S_r$ property on simplicial complexes.
      We prove that a simplicial complex is sequentially $S_r$ if and only if
      its pure skeletons are all $S_r$, a generalization of Duval's result
      on sequentially Cohen-Macaulay simplicial complexes \cite[Theorem 3.3]{Du-96}.
      We show that a simplicial complex is sequentially $S_2$
      if and only if its pure $i$-skeletons are connected for all $i\ge 1$
      and the link of every singleton is sequentially $S_2$.\\

      Recall that a finitely generated graded module $M$ over a Noetherian graded
      $k$-algebra $R$ is said to satisfy the Serre's condition $S_r$
      if \[ \depth M_{\mathfrak p} \ge \min(r, \dim M_{\fp}), \] for all ${\fp} \in
      \Spec(R).$\\

      First we bring the definition of sequentially $S_r$ modules.

      \begin{defn} Let $M$ be a finitely-generated $\mathbb{Z}$-graded module
      over
      a standard graded $k$-algebra $R$ where $k$ is a field. For a
      positive integer $r$ we say that $M$ is sequentially $S_r$ if there
      exists a finite filtration
      $$0=M_0\subset M_1\subset \ldots \subset M_t=M$$
      of $M$ by graded submodules $M_i$ satisfying the two conditions

      (a) Each quotient $M_i/M_{i-1}$ satisfies the $S_r$ condition of Serre.

      (b) ${\rm dim}(M_1/M_0) < {\rm dim}(M_2/M_1) < \ldots < {\rm
      dim}(M_t/M_{t-1}).$
      \end{defn}

      We say that a simplicial complex $\Delta$ on $[n]=\{1,\ldots,n\}$
      is sequentially $S_r$ (over a field $k$) if its face ring
      $k[\Delta]=k[x_1,\ldots,x_n]/I_\Delta$, as a module over $R=k[x_1,\ldots,x_n]$,
      is sequentially $S_r$.

      This is a natural generalization of a $S_r$ simplicial complex,
      i.e., when $k[\Delta]$ satisfies the $S_r$ condition of Serre.

      Since $k[\Delta]$ is a reduced ring, it always satisfies $S_1$
condition.
      Thus, throughout this paper we will always deal with $S_r$ for
$r\geq 2$.

      Using a result of Schenzel \cite[Lemma 3.2.1]{Sch-82} and
      Hochster's formula on local cohomology modules,
      N. Terai has formulated the analogue of Reisner's criterion for Cohen-Macaulay
      simplicial complexes in the case of $S_r$ simplicial complexes
      \cite[page 4, following Theorem 1.7]{Ter-07}. According to this formulation,
      a simplicial complex $\Delta$ of dimension $d-1$ is $S_r$ if
      and only if for all $-1\leq i\leq r-2$ and all $F\in \Delta$ (including
      $F=\emptyset$) with $\#F\leq d-i-2$ we have $\widetilde{H}_i({\rm lk}_\Delta
      F;k)=0$, where ${\rm lk}_\Delta F= \{G\in\Delta: F\cup G\in \Delta, F\cap G
      =\emptyset\}$. For $i=-1$ the vanishing condition is equivalent to
      purity of $\Delta$ and for $i=0$ it is equivalent to the connectedness of
      ${\rm lk}_\Delta F$ \cite[page 4 and 5]{Ter-07}.

      By this characterization of $S_r$ simplicial complexes it follows that
      the $S_r$ property carries over links.

      \begin{lem} \label{lem:linkage} Let $\Delta$ be a $d-1$ dimensional
      simplicial complex
      which satisfies the $S_r$ condition. Then for each $F\in \Delta$ the simplicial
      complex $\lk_{\Delta}F$  also satisfies $S_r$.
      \end{lem}
      \begin{proof}
      Let $\#F=j$, then $\dim \lk_{\Delta}F\leq d-j-1$. By the above characterization of
      $S_r$ simplicial complexes,
      it suffices to show that for all $i\leq r-2$ and  every $G \in \lk_{\Delta}F$, with $\#G
      \leq d-j-i-2$, the reduced homology module
      $\widetilde{H}_i(\lk_{\lk_{\Delta}F}G;k)$ is zero. This follows from the facts that
      ${\rm lk}_{{\rm lk}_\Delta F}(G) =
      {\rm lk}_\Delta (F\cup G)$, $\#(F\cup G) \le d-i-2$ and $\Delta$ is $S_r$.
      \end{proof}

      Recall that a relative simplicial complex is a pair of simplicial
      complexes $(\Delta, \Gamma)$ where $\Gamma$ is a subcomplex of
      $\Delta$. For a relative simplicial complex $(\Delta, \Gamma)$
      define $I_{\Delta,\Gamma}$ to be the ideal in $k[\Delta]$ generated
      by the monomials $x_{i_1}x_{i_2}\ldots x_{i_s}$ with $F=\{i_1,\ldots ,i_s\}\in \Delta\setminus \Gamma$. A
      relative simplicial complex is said to be $S_r$ if $I_{\Delta,
      \Gamma}$ is $S_r$ as a module over $R=k[x_1,\ldots,x_n]$.  Let
      $\Delta^*_i$ be the subcomplex of $\Delta$ generated by its
      $i$-dimensional facets.  Following  \cite[Appendix
      II]{Bj-Wa-We-09}, it turns out that $\Delta$ is sequentially
      $S_r$ if and only if the relative simplicial complex $(\Delta^*_i,
      \Delta^*_i\cap (\Delta^*_{i+1}\cup\ldots \cup \Delta^*_{{\rm
      dim}(\Delta)}) )$ is $S_r$ for all $i$.

      For a relative simplicial complex $(\Delta, \Gamma)$, let
      $\widetilde{H}_i(\Delta, \Gamma; k)$ denote the $i$th reduced
      relative homology group of the pair $(\Delta, \Gamma)$ over $k$
      (see \cite[Chapter III, \S 7]{St-95}). Reisner's
      criterion for Cohen-Macaulayness of a relative simplicial complex
      is similar to the one for a simplicial complex \cite[Chapter III, Theorem
      7.2]{St-95}. Likewise, in an exact analogy, Terai's formulation for
      a $S_r$ simplicial complex carries over for the relative case. In other
      words, a relative simplicial complex $(\Delta,\Gamma)$ is $S_r$ if
      and only if $\widetilde{H}_i({\rm lk}_\Delta F, {\rm lk}_\Gamma F; k)
      =0$ for all $-1\leq i\leq r-2$ and all $F\in \Delta$ (including
      $F=\emptyset$) with $\#F\leq d-i-2$, where $d-1={\rm
      dim}(\Delta)$.

      For a relative simplicial complex $(\Delta, \Gamma)$, as an $R$-module,
      $I_{\Delta,\Gamma}$ only depends on the difference $\Delta
      \setminus \Gamma$ (see the remarks following \cite[Chapter III, Theorem
      7.3]{St-95}). In particular, if $\Delta^{(i)}$ is the
      $i$-skeleton of $\Delta$ and $\Delta^{[i]}$ is the pure
      $i$-skeleton of $\Delta$, then
      $$\Delta^*_i  \setminus  (\Delta^*_i\cap
      \Delta^*_{i+1}\cup\ldots \cup \Delta^*_{{\rm dim}(\Delta)}) =
      \Delta^{[i]} \setminus (\Delta^{[i+1]})^{(i)}.$$
      Duval makes the
      above observation and proves that the relative simplicial complex
      $\Delta^{[i]} \setminus (\Delta^{[i+1]})^{(i)}$ is Cohen-Macaulay for
      all $i$ if and only if every pure $i$-skeleton $\Delta^{[i]}$ is
      Cohen-Macaulay \cite[Section 3]{Du-96}. We follow his proof
      step by step to show that the same result
      is true if we replace the Cohen-Macaulay property with $S_r$.
      To do this we need some preliminary results.\\

      It is known that if $\Delta$ is a Cohen-Macaulay simplicial
      complex, then so is $\Delta^{(i)}$, the
      $i$-skeleton of $\Delta$. We generalize this result for the
      property $S_r$.

      \begin{prop}\label{prop:Skeleton_S_r}
      If $\Delta$ satisfies Serre's condition $S_r$, then $\Delta^{(i)}$
      satisfies this condition $(2 \leq r \leq i+1)$.

      \end{prop}

      \begin{proof} We check Terai's criterion for $S_r$ simplicial complexes.
      To prove the assertion for $\Delta^{(i)}$ we use
      induction on $r\geq 2$. Assume that $\Delta$ satisfies Serre's
      condition $S_2$. Then $\Delta$ is pure, hence $\Delta^{(i)}$ is pure. Furthermore,
      for $F\in \Delta$ with $\#F \leq d-2$  ${\rm lk}_{\Delta}F$ is connected.
      Let $F\in \Delta^{(i)}$ and $\#F\leq i- 1$.
      It is enough to show that ${\rm lk}_{\Delta^{(i)}}F$ is connected, or
      equivalently, path connected. Let $\{u\}, \{v\} \in {\rm lk}_{\Delta^{(i)}}F$.
      Then $\{u\}, \{v\} \in {\rm lk}_{\Delta}F$ which is connected. Hence,
      there exists a sequence of vertices of $\Delta$, $u_0=u, u_1, \ldots, u_t=v$, such that $\{ u_j,
      u_{j+1}\} \in {\rm lk}_{\Delta}F$, $j=0,\ldots, t-1$.  Thus, $\{
      u_j, u_{j+1}\}\cap F=\emptyset$ and $\{ u_j, u_{j+1}\}\cup F \in
      \Delta$. Since $\#(\{ u_j, u_{j+1}\}\cup F)\leq i+1$, $\{ u_j,
      u_{j+1}\}\cup F \in \Delta^{(i)}$ and hence $\{ u_j, u_{j+1}\}\in {\rm
      lk}_{\Delta^{(i)}}F$.

      Now assume that $\Delta$ satisfies Serre's condition $S_r$ for
      $r>2$. Then $\Delta$ satisfies Serre's condition $S_j$ for $j=1,
      \ldots, r$. Thus by induction hypothesis $\Delta^{(i)}$ satisfies
      Serre's condition $S_j$ for $j=1, \ldots, r-1$. Therefore,  for
      $q\leq r-3$ and $F\in \Delta^{(i)}$ with $\#F\leq i-q-1$,
      $\widetilde{H}_q({\rm lk}_{\Delta^{(i)}}F; k) =0$. Thus it remains to
      show that for $\#F\leq i-r+1$, $\widetilde{H}_{r-2}({\rm
      lk}_{\Delta^{(i)}}F; k)  =0$. To prove this, since $ {\rm
      lk}_{\Delta^{(i)}}F \subset  {\rm lk}_{\Delta}F$, it is enough to show
      that for $q\le r-1$, any $q$-dimensional face $H$ of ${\rm lk}_{\Delta}F$ lies in
      ${\rm lk}_{\Delta^{(i)}}F$. But $\#(H\cup F) \le i+1$, and
      hence, $H\cup F \in \Delta^{(i)}$, and consequently, $H\in {\rm
      lk}_{\Delta^{(i)}}F$.

      \end{proof}

      Now we adopt Duval's results for the case of sequentially $S_r$
      simplicial complexes.

      \begin{lem}\label{lem:relative equality} (see \cite[Lemma 3.1]{Du-96}).
      Let $F$ be a
      face of a $(d-1)$-dimensional simplicial complex $\Delta$ and let
      $\Gamma$ be either the empty simplicial complex or a $S_r$
      simplicial complex of the same dimension as $\Delta$. Then
      $$\widetilde{H}_i({\rm lk}_\Delta F; k) = \widetilde{H}_i({\rm
      lk}_\Delta F; {\rm lk}_\Gamma F; k)$$ for all $i\leq r-2$ and all $F\in
      \Delta$ with $\#F \leq d-i-2$.
      \end{lem}

      \begin{proof}
      The proof is the same as the proof of the similar lemma of Duval
      \cite[Lemma 3.1]{Du-96}. If ${\rm lk}_\Gamma F$ is an empty set,
      then the equality is obvious. Otherwise one only needs to change the range of the
      index $i$ with the one given above, impose the condition on $\#F$
      and replace Cohen-Macaulay property with $S_r$.
      \end{proof}

      \begin{cor} (see \cite[Corollary 3.2]{Du-96}). Let
      $\Delta$ be a simplicial complex, and let $\Gamma$ be either the
      empty simplicial complex or a $S_r$ simplicial complex of the same
      dimension as $\Delta$. Then $\Delta$ is $S_r$ if and only if
      $(\Delta, \Gamma)$ is relative $S_r$.
      \end{cor}

      \begin{proof}
      Similar to the corresponding corollary by Duval \cite[Corollary
      3.2]{Du-96}, it follows from Lemma \ref{lem:relative equality}.
      \end{proof}

      \begin{thm} \label{thm:pure skeletons} (see \cite[Theorem 3.3]{Du-96}).
      Let
      $\Delta$ be a $(d-1)$-dimensional simplicial complex. Then $\Delta$
      is sequentially $S_r$ if and only if its pure $i$-skeleton
      $\Delta^{[i]}$ is $S_r$ for all $-1\leq i \leq d-1$.
      \end{thm}

      \begin{proof}
      The proof is the same as the one given by Duval \cite[Theorem
      3.3]{Du-96}. The only item needed is that each $i$-skeleton of a $S_r$
      simplicial complex is again $S_r$ for all $i$. But this is proved in
      Proposition \ref{prop:Skeleton_S_r}.
      \end{proof}

      The following is an immediate bi-product of this theorem.

      \begin{cor} A simplicial complex $\Delta$ is $S_r$ if and only
      if it is sequentially $S_r$ and pure.
      \end{cor}

      \begin{proof}
      Since $S_r$ condition implies purity, one implication is
      clear. Assume that $\Delta$ is pure and ${\rm dim}\Delta =d-1$.
      Then $\Delta^{[d-1]}=\Delta$
      and the assertion follows by Theorem \ref{thm:pure skeletons}.
      \end{proof}

      \begin{rem}
      Some authors define a simplicial complex to be sequentially Cohen-
      Macaulay if its pure $i$-skeleton is Cohen-Macaulay for all $i$.
      Likewise, we might take a similar statement as the definition of
      sequentially $S_r$ simplicial complexes. But we preferred to begin
      with the algebraic definition given in this section and prove that
      both definitions are equivalent.
      \end{rem}

      We end this section with the following characterization of
      sequentially $S_2$ simplicial complexes which will be used in
      the last section.

      \begin{thm}\label{thm:localcriteria}
Let $\Delta$ be a simplicial complex with vertex set $\V$. Then
$\Delta$ is sequentially $S_2$ if and only if the following
conditions hold:

\begin{itemize}

\item[(i)] $\Delta^{[i]}$ is connected for all $i\ge 1$

\item[(ii)] $\lk_\Delta(x)$ is sequentially $S_2$ for all $x\in \V$

\end{itemize}

\end{thm}

\begin{proof}

Let $\Delta$ be a sequentially $S_2$. Then $\Delta^{[i]}$ is $S_2$
for all $-1\le i\le d-1$. Thus $\Delta^{[i]}$ is connected for all
$i\ge 1$. On the other hand $\lk_{\Delta^{[i]}}(x)$ is $S_2$ for all
$-1\le i\le d-1$ and so
$\lk_{\Delta^{[i+1]}}(x)=(\lk_\Delta(x))^{[i]}$ is $S_2$ for all
$-1\le i\le d-2$. Therefore $\lk_\Delta(x)$ is sequentially $S_2$.

Now let $\Delta$ satisfies the conditions (i) and (ii). Since
$\lk_\Delta(x)$ is sequentially $S_2$ for all $x\in\V$ we have that
$(\lk_\Delta(x))^{[i]}=\lk_{\Delta^{[i+1]}}(x)$ is $S_2$ for all
$-1\le i\le d-2$ and so $\lk_{\Delta^{[i]}}(x)$ is $S_2$ for all
$-1\le i\le d-1$. Now the connectedness of $\Delta^{[i]}$ for $i\ge
1$ implies that $\Delta^{[i]}$ is $S_2$ for all $-1\le i\le d-1$.
Indeed, for $F\ne \emptyset$ and $x\in F$, $\lk_{\Delta^{[i]}}F =
\lk_{\lk_{\Delta^{[i]}}(x)}G$ for $G=F\setminus \{x\}$. Therefore
$\Delta$ is sequentially $S_2$.

\end{proof}

      \section{Alexander dual of sequentially $S_r$ simplicial complexes}

      In this section we show that a simplicial complex is sequentially $S_r$ if and
      only if the minimal free resolution of the Stanley-Reisner ideal of
      its Alexander dual is componentwise linear in the first $r$ steps.
      This result resembles a result of Herzog and Hibi  \cite[Proposition 1.5]{He-Hi-99}
      on sequentially Cohen-Macaulay simplicial complexes. And, our proof
      would be a modification of the sequentially Cohen-Macaulay case together with
      an application of Theorem \ref{thm:pure skeletons}.

      We first adopt the following definitions from \cite[Definition 3.6]{Ya-2000}
      and \cite[\S 1]{He-Hi-99}.

      Consider $R=k[x_1,\ldots, x_n]$ with ${\rm deg}(x_i)=1$ for all $i$. If
      $I$
      is a homogenous ideal of $R$ and $r\ge 1$, then $I$ is said to
      be linear in the first $r$ steps, if for some integer $d$, $\beta_{i,i+t}(I)=0$
      for all $0\le i < r$ and $t\neq d$. We write $I_{<j>}$ for the ideal
      generated by all homogenous polynomials of degree $j$ belonging to
      $I$. We say that a homogenous ideal $I\subset R$ is componentwise
      linear if $I_{<j>}$ has a linear resolution for all $j$. The ideal
      $I$ is said to be {\it componentwise linear in the first $r$ steps}
      if for all $j\ge 0$, $I_{<j>}$ is linear in the first $r$
      steps. A simplicial complex $\Delta$ on $[n]$ is said to be
      linear in the first $r$ steps, componentwise linear and
      componentwise linear in the first $r$ steps, if $I_\Delta$
      satisfies either of
      these properties, respectively.

      Now let $I\subset R$ be an ideal generated by squarefree monomials.
      Then for each degree $j$ we write $I_{[j]}$ for the ideal generated
      by the squarefree monomials of degree $j$ belonging to $I$. We say
      that $I$ is squarefree componentwise linear if $I_{[j]}$ has linear
      resolution for all $j$. The ideal $I$ is said to be {\it squarefree
      componentwise linear in the first $r$ steps} if $I_{[j]}$ has a
      resolution which is linear in the first $r$ steps for all $j$.

      Below we adopt a result of Herzog and Hibi \cite[Proposition
      1.5]{He-Hi-99} for the case of componentwise linearity in the first $r$
      steps.

      \begin{prop}\label{prop:Square_Component}
      Let $I$ be a squarefree monomial ideal in $R$. Then $I$ is
      componentwise linear in the first $r$ steps if and only if $I$ is
      squarefree componentwise linear in the first $r$ steps.
      \end{prop}
      \begin{proof}
      The proof is the same as \cite[Proposition 1.5]{He-Hi-99} with just
      a restriction on the index $i$ used in the proof of Herzog and
      Hibi. Here we need to assume  that $i < r$. Also one needs to
      observe that when $I$ has a linear resolution in the first $r$ steps,
      the ideal $\fm I$  has a linear resolution in the first $r$ steps too.
      Here $\fm =(x_1,\ldots,x_n)$ is the irrelevant maximal ideal.

      \end{proof}

      We may now generalize a result of Herzog and Hibi \cite[Theorem
      2.1]{He-Hi-99}. As we already mentioned, Yanagawa and Terai proved
      that $\Delta$ is $S_r$ if and only if $I_\Delta^\vee$ is
      linear in the first $r$ steps.

      \begin{thm}
      The Stanley-Reisner ideal of $\Delta$ on $[n]$ is componentwise linear in the
      first $r$ steps if and only if $\Delta^\vee$, the Alexander dual of
      $\Delta$, is sequentially $S_r$.
      \end{thm}
      \begin{proof}
      The proof is an adaptation of the proof of part (a) of \cite[Theorem
      2.1]{He-Hi-99} with the following additional remarks: Let $I =
      I_\Delta$. Then by Proposition \ref{prop:Square_Component}, $I$ is
      squarefree componentwise linear in the first $r$ steps if and
      only if $I$ is componentwise linear in the first $r$ steps. By
      \cite[Corollary 3.7]{Ya-2000} for every $j$, $I_{[j]}$ is linear in the first
      $r$ steps if and only if $(\Delta^\vee)^{[n-j-1]}$ is $S_r$.
      Therefore, $I$ is componentwise linear in the first $r$ steps
      if and only if $(\Delta^\vee)^{[q]}$ is $S_r$ for every $q$.
      But by Theorem \ref{thm:pure skeletons}, this is equivalent to the
      sequentially $S_r$ property for $\Delta^\vee$.
      \end{proof}
      Van Tuyl and Villarreal \cite[Theorem 3.8 (a)]{V-V-08} state the
      dual version of \cite[Theorem 2.1]{He-Hi-99} for sequentially
      Cohen-Macaulay simplicial complexes. Dualizing the statement of the
      above
      theorem we get a similar generalization for sequentially $S_r$
      simplicial complexes.

      \begin{cor}\label{Cor:Square_Component}
      A simplicial complex $\Delta$ is sequentially $S_r$  if and only if
      the Stanley-Reisner ideal of the Alexander dual of $\Delta$ is
      componentwise linear in the first $r$ steps.
      \end{cor}

      \section{Some characterizations of sequentially $S_r$ cycles and
      sequentially $S_2$ bipartite graphs}

      In this section, we provide some applications of the results of the previous
      sections. We first classify sequentially $S_r$ cycles and show that a cycle $C_n$ is
      sequentially $S_2$ if and only if $n$ is odd and no cycles are
      sequentially $S_3$ except those which are Cohen-Macaulay, i.e.,
      $C_3$ and $C_5$. This generalizes a result of Francisco and Villarreal \cite[Theorem 4.1]{Fr-Va-07}.
      Then we generalize some results of Van Tuyl and Villarreal \cite{V-V-08}
      for sequentially $S_r$ graphs.
      We also extend a result of Francisco and H\`{a} \cite[Theorem
      4.1]{Fr-Ha-08} on graphs with whiskers. Then we generalize a result of Van Tuyl
      \cite[Theorem  2.10]{V-09} who proves that a bipartite graph is
      vertex decomposable if and only if it is sequentially Cohen-Macaulay. We prove
      that a bipartite graph is vertex decomposable if and only if it is
      sequentially $S_2$. This result also generalizes the authors'
      result which states that a bipartite graph is Cohen-Macaulay if and
      only if it is $S_2$ \cite[Theorem 1.3]{H-Y-Z-09}.
      Woodroofe \cite[Theorem 1]{Wo-08} proved that a graph with no chordless cycles other that
      cycles of length $3$ and $5$ is sequentially Cohen-Macaulay. We
      provide some results which extend this statement for $S_2$ property.
      In particular they imply that any graph with no chordless even cycle
      is sequentially $S_2$.

      At the end of this section we pose two questions on
      sequentially $S_r$ property of join of two simplicial
      complexes and topological invariance of $S_r$, respectively.\\

      Recall that to a simple graph $G$ one associates a simplicial
      complex $\Delta_G$ on $V(G)$, the set of vertices of $G$, whose faces
      correspond to the independent sets of
      vertices of $G$. A graph $G$ is said to be $S_r$ if $\Delta_G$ is a $S_r$
      simplicial complex. Likewise, $G$ is Cohen-Macaulay, sequentially
      Cohen-Macaulay and shellable if $\Delta_G$ satisfies either of these
      properties, respectively. We adopt the definition of shellability in
      the nonpure sense of Bj\"{o}rner-Wachs \cite{B-W-96}.

      We also recall the definition of a vertex decomposable
      simplicial complex. A simplicial complex $\Delta$ is
      recursively defined to be vertex decomposable if it is either
      a simplex or else has some vertex $v$ so that
      \begin{itemize}
      \item[(1)] Both $\Delta\setminus \{v\}$ and $\lk_\Delta(v)$
      are
      vertex decomposable, and
      \item[(2)] No face of $\lk_\Delta(v)$ is a facet of $\Delta\setminus
      \{v\}$.
      \end{itemize}

      The notion of vertex decomposability was introduced in the
      pure case by Provan and Billera \cite{PB} and was extended
      to nonpure complexes by Bj\"{o}rner and Wachs \cite{B-W-96}.\\

      Sequentially Cohen-Macaulay cycles have been characterized by
Francisco and Van Tuyl \cite[Proposition 4.1]{Fr-Va-07}. They are
just $C_3$ and $C_5$. Woodroofe has given a more geometric proof for
this result \cite[Theorem 3.1]{Wo-08}. In \cite[Proposition
1.6]{H-Y-Z-09} it is shown that the only $S_2$ cycles are $C_3$,
$C_5$ and $C_7$. We now generalize this result and prove that the
odd cycles are sequentially $S_2$ and they are the only sequentially
$S_2$ cycles.

\begin{thm}\label{cycles}

The cycle $C_n$ is sequentially $S_2$ if and only if $n$ is odd.

\end{thm}

\begin{proof}
Let $n=2k$.  Then $\Delta=\Delta_{C_n}$ has only two facets of
dimension $2k-1 \ (=\dim \Delta)$, namely, $\{1,3,\ldots,2k-1\}$ and
$\{2,4,\ldots,2k\}$. Thus $\Delta^{[k-1]}$ is the union of two
disjoint $(k-1)$-simplices. In particular, it is disconnected,
contradicting Terai's criterion for $\Delta^{[k-1]}$ to be $S_2$.

Let $n$ be an odd integer. To show that $C_n$ is sequentially $S_2$,
by Theorem \ref{thm:localcriteria}, we need to show that
$\Delta^{[i]}$ is connected for all $i\ge 1$ and $\lk_\Delta(x)$ is
sequentially $S_2$ for all $\{x\} \in \Delta$. But $\lk_\Delta(x)$
is the simplicial complex of a $(n-3)$-path $P_{n-3}$ and so it is
sequentially Cohen-Macaulay, e.g., by \cite[Theorem 3.2]{Fr-Va-07}.

To prove the connectedness of the pure $i$-skeleton of $\Delta$ we
first claim that for each $E\in\Delta$ with $\dim E=1$, there exists
$F\in\Delta$ such that $\dim F=\dim \Delta$ and $E\subseteq F$. We
know that $$1+\dim\Delta=\max\{\# F|F\,\,\mbox{is an independent set
of}\,\, C_n\}=(n-1)/2.$$ Let $E=\{x,y\}\in\Delta$. The vertices $x$
and $y$ divide $C_n$ into two disjoint parts $A_1$ and $A_2$. Since
$n$ is odd, one of these parts has odd number of vertices and the
other has even number of vertices. Let $A_1$ has odd number of
vertices and $\#A_1 =t$. For $i=1,2$, let $F_i$ be the face of
$\Delta$ obtained from $A_i$ excluding the end points of $A_i$ such
that it has the maximum cardinality. Then $\#F_1 = (t-2)/2$ and
$\#F_2 = (n-t-3)/2$. Let $F=F_1\cup F_2\cup E$.  Then $\#F=(n-1)/2$
and $E\subseteq F$. Therefore the claim holds.

Let $E\in \Delta$ with $\dim E=1$. Then there exists $F\in\Delta$
such that $E\subseteq F$ and $\dim F=\dim\Delta$. Thus for each
$i\ge 1$ there exists $H\in\Delta$ such that $\dim H=i$ and
$E\subseteq H\subseteq F$. Therefore $E\subseteq H\in\Delta^{[i]}$.
Thus we have shown that $E\in\Delta$ with $\dim E=1$ if and only if
$E\in\Delta^{[i]}$ for all $i\ge 1$. Therefore $(\Delta^{[i]})^{[1]}
=\Delta^{[1]}$.

On the other hand $\Delta^{[i]}$ is connected if and only if
$(\Delta^{[i]})^{[1]}$ is connected. Therefore it is enough to show
that $\Delta^{[1]}$ is connected. Let $x$ and $y$ be two elements in
$\V$. If $x$ is not adjacent to $y$, then $\{x,y\}\in\Delta$. If $x$
is adjacent to $y$, then there exists $z\in\V$ such that $x$ is not
adjacent to $z$ and $y$ is not adjacent to $z$. Therefore $\{x,z\}$
and $\{y,z\}$ belong to $\Delta$ and hence $\Delta^{[1]}$ is
connected.

\end{proof}

The following result completes the characterization of cycles with
respect to the property $S_r$. We will give two proofs using some
extra data from the algebraic proof of \cite[Proposition
4.1]{Fr-Va-07} and the geometric proof of \cite[Theorem 3.1]{Wo-08}.

\begin{prop}\label{S3cycles}
The cycle $C_n$ is sequentially $S_3$ if and only if $n=3, 5$, i.e.,
if and only if $C_n$ is Cohen-Macaulay.
\end{prop}

\begin{proof} The first proof: Let $\Delta = \Delta_{C_n}$. Let $n=2t+1 > 7$.
By the proof of \cite[Theorem4.1]{Fr-Va-07}, if $J$ is the Alexander
dual of $I_\Delta$, then $J_{[t+1]}$ fails to have linear resolution
in the first $3$ steps. Hence $C_n$ is not sequentially $S_3$ by
Corollary \ref{Cor:Square_Component}.

For $n=7$, again we follow the proof of \cite[Theorem
4.1]{Fr-Va-07}. In this case, the ideal $J$ is generated in degree
$4$. Thus the resolution of $J$ is the same as the resolution of
$J_{[4]}$. Moreover, the resolution of $J$ is linear in the first
$2$ steps but not in step $3$. Thus by Corollary
\ref{Cor:Square_Component}, $C_7$ is not sequentially $S_3$.

The second proof: Let $n=2t+1 \ge 7$. By \cite[Lemma 3.2]{Wo-08},
$\Delta^{[t-1]}$ is homotopic to the circle $S^1$. Thus
$\widetilde{H}_1(\Delta^{[t-1]};k) = k$. By Hochster's formula on
Betti numbers of simplicial complexes we have ${\rm depth}
(k[\Delta^{[t-1]}]) \le 2$. Thus $k[\Delta^{[t-1]}]$ does not
satisfy $S_3$. Hence $C_n$ is not sequentially $S_3$.

\end{proof}

      We now outline the analogue statements of \cite[Section 3]{V-V-08}
      to conclude the sequentially $S_r$ counterparts of some of the results there.

      \begin{lem} \label{lem:degree one} (see \cite[Lemma 3.9]{V-V-08})
      Let $G$ be a bipartite graph with bipartition $\{x_1,\ldots,x_m\}$
      and $\{y_1,\ldots,y_n\}$. If $G$ is sequentially $S_r$, then there
      is a vertex $v \in V(G)$ with $\deg(v)=1$.
      \end{lem}
      \begin{proof}
      The only modification needed in the proof of \cite[Lemma
      3.9]{V-V-08} is to justify that the kernel of the linear map $f$
      used in that proof, is generated by linear syzygies. But under the
      hypothesis of the lemma, this is proved in Proposition
      \ref{prop:Square_Component}.
      \end{proof}

      For a graph $G$ and a vertex $x\in V_G$,
      the set of neighbors of $x$ will be denoted by $N_G(x)$. For
      $F\in \Delta_G$ we set:
      $$N_G(F)= \cup_{x\in F} N_G(x).$$
      We also use the following notation:
      $$N_G[x] = \{x\}\cup N_G(x),$$
      $$N_G[F] = F \cup N_G(F).$$

      The following lemma gives a recursive procedure to check if a graph
      fails to be sequentially $S_r$. It is a sequentially $S_r$ version of
      \cite[Theorem 3.3]{V-V-08}.

      \begin{lem} \label{lem:sequential-graph}
      Let $G$ be  graph and $x$ a vertex of $G$. Let $G'=G\setminus N_G[x]$.
      If $G$ is sequentially $S_r$, then $G'$ is also
      sequentially $S_r$.
      \end{lem}
      \begin{proof}
      The proof is identical with that of \cite[Lemma 3.3]{V-V-08}.
      We only need to use Theorem \ref{thm:pure
      skeletons} and Lemma \ref{lem:linkage} instead of their sequentially
      Cohen-Macaulay and Cohen-Macaulay
      counterparts, respectively.
      \end{proof}

      Lemma \ref{lem:sequential-graph} could be extended further.

      \begin{cor}
      Let $G$ be a graph  which is sequentially $S_r$. Let $F$ be an
      independent set in $G$. Then the graph $G'=G\setminus N_G[F]$
      is sequentially $S_r$.

      \end{cor}
      \begin{proof}
      This follows by repeated applications of Lemma
      \ref{lem:sequential-graph}.
      \end{proof}

      The following generalizes a result of Francisco and H\`{a} \cite[Theorem
      4.1]{Fr-Ha-08} which is also proved by Van Tuyl and Villarreal  by a
      different method (see \cite[Corollary 3.5]{V-V-08}).

      Recall that for a subset $S=\{y_1,\ldots, y_m\}$ of a graph $G$, the graph $G\cup
      W(S)$ is obtained from $G$ by adding an edge (whisker) $\{x_i,y_i\}$
      to $G$ for all $i=1,\ldots, m$, where $x_1,\ldots,x_m$ are new
      vertices.

       \begin{cor}\label{whisker}
      Let  $S\subset V(G)$ and suppose that the graph $G\cup W(S)$ is sequentially
      $S_r$, then $G\setminus S$ is sequentially $S_r$.
      \end{cor}
      \begin{proof}
      This also follows by repeated applications of Lemma
\ref{lem:sequential-graph}.
      \end{proof}

      Van Tuyl \cite[Theorem 2.10]{V-09} has proved that
      a bipartite  graph is vertex decomposable if and only if it is sequentially
      Cohen-Macaulay. We now generalize this result. Observe that our
      result also generalizes the authors' result which states that a
      bipartite graph is  Cohen-Macaulay if and only if it is $S_2$
      \cite[Theorem 1.3]{H-Y-Z-09}.

      First we need a more general result.

\begin{thm}\label{chordlessgeneral}
Let $G=(V,E)$ be a graph. Suppose $H =G \setminus N_G[F]$ satisfies
one of the conditions (i), (ii), and (iii)
 for any $F \in \Delta _G$ which is not a facet:
\begin{enumerate}
\item[(i)] $H$ has no chordless even cycle.
\item[(ii)]  $H$ has a simplicial vertex, i.e., for some $z\in
V(H)$, $N_H[z]$ is a complete graph.
\item[(iii)]  For some $t\ge 2$, $H$ has a chordless $(2t+1)$-cycle
which has  $t$ independent vertices of degree 2  in $H$.
\end{enumerate}
Then $G$ is sequentially $S_2$.
\end{thm}

\begin{proof}
We prove the theorem by induction on $n$, the number of vertices of
$G$. The assertion holds for $n \le 3$. Now we assume that $n \ge
4$. Set $\Delta=\Delta_G$ and let $x \in V(G)$. Observe that $G'=
G\setminus N_G[x]$ satisfies the statement of the theorem. Thus it
is sequentially $S_2$ by the induction hypothesis. Hence by
\cite[Lemma 2.5]{V-V-08}, $\lk_G(x) = \Delta_{G'}$ is sequentially
$S_2$. Thus by Theorem \ref{thm:localcriteria} it is enough to show
that $\Delta ^{[i]}$ is connected for $1 \le i \le \dim \Delta$ . We
show that for any $X, Y \in \Delta$ with $\dim X= \dim Y=i$, there
is a chain $X=X_0,X_1, \dots , X_s=Y$ of $i$-faces of $\Delta$ such
that $X_{j-1} \cap X_{j} \ne \emptyset$ for $j=1,2, \dots , s$. We
may assume $X \cap Y = \emptyset$. For simplicity we set
$X=\{x_1,x_2, \dots, x_{i+1}\}$ and $Y= \{y_1,y_2,\dots, y_{i+1}\}$.

We assume that the condition (i) is satisfied for $G$. Set $B=G_{X
\cup Y}$, the restriction of $G$ to $X \cup Y$. Then $B$ is a
bipartite graph on the partition $X \cup Y$. Since $B$ is bipartite,
$B$ has no odd cycle. Since $B$ has no (chordless) even cycle by the
condition (i), $B$ is a forest. Then there exists a vertex with
degree 0 or 1 in $B$. We may assume that $x_1$ is such a vertex and
that $x_1$ is adjacent  at most to $y_1$. Set $X_1=\{x_1,y_2,\dots,
y_{i+1}\}$.  Then $X, X_1, Y$ is a desired chain.

We assume that the condition (ii) is satisfied for $G$. Then using
the hypotheses for $F=\emptyset$, there is a simplicial vertex $z$
in $G$. Assume $z \not\in X \cup Y$. Since $z$ is simplicial, $z$ is
adjacent to at most  one vertex in $X$. We may assume that  $z$ is
not  adjacent to $x_2, \dots, x_{i+1}$. Similarly, we may assume
that $z$ is not adjacent to $y_2,   \dots, y_{i+1}$. Set
$X_1=\{z,x_2,\dots, x_{i+1}\}$ and $X_2= \{z,y_2,\dots, y_{i+1}\}$.
Then $X, X_1, X_2, Y$ is a desired chain. Assume $z \in  X \cup Y$.
We may assume $z=y_1 \in  Y$. Then $X, X_1, Y$ is a desired chain.

Next we assume that the condition (iii) is satisfied for $G$. Then
for some $t \ge 2$ there exists a chordless $(2t+1)$-cycle $C$ which
has  $t$ vertices of degree 2 in $G$  which are independent in $G$.
Let $\{z_1, z_2, \dots ,z_t \} \subset  V(C)$  be an independent set
of vertices of $G$ such that $\deg _G z_j=2$ for $j=1,2, \dots, t.$

Case I. $X \cup Y \subset V(C).$ As in the case that the condition
(i) is satisfied, $B=G_{X \cup Y}$ is a bipartite graph on the
partition $X \cup Y$. Since $C$ has no chord, $B$ is a disjoint
union of paths. Then we can construct a desired chain as in the
above case.

Case II. $X \subset V(C)$, and $Y \cap (V(G) \setminus V(C)) \ne
\emptyset$. We may assume that $y_1 \in V(G) \setminus V(C)$. Note
that $i+1 \le t.$ Set $Y_1=\{y_1, z_2, \dots ,z_{i+1}  \}$ and $Z=
\{z_1, z_2, \dots ,z_{i+1} \}$. Then $Y, Y_1, Z $ is a chain with $Y
\cap Y_1 \ne \emptyset$, $Y_1 \cap Z  \ne \emptyset$. Between $Z$
and $X$,  we have a desired chain as in Case I. Hence we have a
desired chain between $X$ and $Y$.

Case III. $X \cap (V(G) \setminus V(C)) \ne \emptyset$ and $Y
\subset V(C)$. As in Case II.

Case IV. $X \cap (V(G) \setminus V(C)) \ne \emptyset$ and $Y \cap
(V(G) \setminus V(C)) \ne \emptyset$. As in Case II, we can
construct desired chains between $X$ and $Z$ and between $Y$ and
$Z$. Thus we have a desired chain between $X$ and $Y$ via $Z$.

\end{proof}

\begin{cor}\label{chordlessclass}
Let $G$ be a bipartite graph. The following conditions are
equivalent:
\begin{enumerate}
\item[(i)]$G$ is vertex decomposable.
\item[(ii)]$G$ is shellable.
\item[(iii)]$G$ is sequentially Cohen-Macaulay.
\item[(iv)]For any $F \in \Delta _G$ (including $\emptyset$) which is not a facet,
there is a vertex $v \in  G \setminus N_G[F]$  such that $\deg_{G
\setminus N_G[F]}(v) \le 1$.
\item[(v)]$G$ is sequentially $S_2$.
\end{enumerate}

\end{cor}

\begin{proof}
$(i) \Rightarrow (ii)$: Follows from \cite[Theorem 11.3]{B-W-97}.

$(ii) \Rightarrow (iii)$: Follows from \cite[Chap. III, \S
2]{St-95}.

$(iii) \Rightarrow (i)$: Follows from \cite[Theorem 2.10]{V-09}.

$(iii) \Rightarrow (iv)$:
For $F \in \Delta _G$ which is not a facet, set $H:=G \setminus N_G[F]$.
If $H$ has no edge, every vertex $v \in V(H)$ satisfies  $\deg_{H}(v) = 0$.
Therefore, we may assume that the bipartite graph $H$ has an edge.
Applying  \cite[Theoerm 3.3]{V-V-08} repeatedly, we know $H$ is sequentially Cohen-Macaulay.
Hence by \cite[Lemma 3.9]{V-V-08} there is a vertex $v \in V(H)$  such that $\deg_{H}(v) = 1$.

$(iii) \Rightarrow (v)$: This is trivial.

$(iv) \Rightarrow (v)$: Suppose for any $F \in \Delta _G$ which is
not a facet, there is a vertex $v \in H:= G \setminus N_G[F]$  such
that $\deg_{H}(v) \le 1$. Hence $N_{H}[v]$ is the 1-complete graph
or the 2-complete graph. This means $v$ is a simplicial vertex in
$H$. By Theorem \ref{chordlessgeneral} (ii), $G$ is sequentially
$S_2$.

$(v) \Rightarrow (ii)$:  The proof is by induction on the number of
vertices of $G$. Now let $G$ be a sequentially $S_2$ graph. By Lemma
\ref{lem:degree one} there exists a degree one vertex $x_1 \in
V(G)$. Assume that $N_G(x_1)=\{y_1\}$. Let $G_1=G\setminus N_G[x_1]$
and $G_2=G\setminus N_G[y_1].$ By Lemma \ref{lem:sequential-graph}
both of these graphs are sequentially $S_2$, hence by the induction
hypothesis they are both shellable. Therefore, by \cite[Theorem
2.9]{V-V-08} $G$ is shellable.

\end{proof}

In \cite[Theorem 1.1]{Wo-08} it is shown that a graph $G$ with no
chordless cycles of length other than 3 or 5 is sequentially
Cohen-Macaulay. In the following we extend this result on a larger
class graphs for the sequentially $S_2$ property.

\begin{thm}\label{chordlesswisker}
Let $G$ be a graph. Suppose that a vertex in each chordless even
cycle in $G$ has a whisker. Then $G$ is sequentially $S_2$.
\end{thm}

\begin{proof}
We prove the theorem by induction on $n$, the number of vertices of
$G$. The assertion holds for $n \le 3$. Now we assume $n \ge 4$. Set
$\Delta=\Delta_G$ and let $x \in V$. Since $G\setminus N_G[x]$
satisfies the condition of the theorem, it is sequentially $S_2$ by
the induction hypothesis. Hence by Theorem \ref{thm:localcriteria}
it is enough to show that $\Delta ^{[i]}$ is connected for $1 \le i
\le \dim \Delta$. We show that for any $X, Y \in \Delta$ with $\dim
X= \dim Y=i$, there is a chain $X=X_0,X_1, \dots , X_s=Y$ of
$i$-faces of $\Delta$ such that $X_{j-1} \cap X_{j} \ne \emptyset$
for $j=1,2, \dots , s$. We may assume $X \cap Y = \emptyset$. For
simplicity we set $X=\{x_1,x_2, \dots, x_{i+1}\}$ and $Y=
\{y_1,y_2,\dots, y_{i+1}\}$. Let $x_1$ have a whisker, that is,
there exists $z \in V(G)$ such that $\deg z=1$ and $\{x_1, z\} \in
E$. Assume $z \not\in Y$. Set $X_1=\{z,x_2,\dots, x_{i+1}\}$ and
$X_2= \{z,y_2,\dots, y_{i+1}\}$. Then $X, X_1, X_2, Y$ is a desired
chain. Assume $z=y_1 \in Y$. Then $X, X_1, Y$ is a desired chain.

Hence we may assume that no vertex in $X \cup Y$ has a whisker in
$G$. Set $B=G_{X \cup Y}$, the restriction of $G$ to $X \cup Y$.
Then $B$ is a bipartite graph on the partition $X \cup Y$. Since $B$
is bipartite,  $B$ has no odd cycle. Since any vertices in $X \cup
Y$ do not have a whisker in $G$, $B$ has no (chordless) even cycle.
Hence $B$ is a forest. Then there exists a vertex with degree 0 or 1
in $B$. We may assume that $x_1$ is such a vertex and that $x_1$ is
connected at most to $y_1$. Set $X_1=\{y_1,x_2,\dots, x_{i+1}\}$.
Then $X, X_1, Y$ is a desired chain.

\end{proof}

\begin{cor}\label{chordless}

If $G$ is a graph with no chordless even cycle, then $G$ is
sequentially $S_2$.

\end{cor}

\begin{rem}

Corollary \ref{chordless} gives an alternative proof for the fact
that any odd cycle is sequentially $S_2$, the significant part of
Theorem \ref{cycles}.

\end{rem}

We end this section by proposing two questions.\\

Let $\Delta$ and $\Gamma$ be two simplicial complexes over disjoint
vertex sets. In \cite{STY} it is shown that $\Delta \ast \Gamma$ is
sequentially Cohen-Macaulay if and only if $\Delta$ and $\Gamma$ are
both sequentially Cohen-Macaulay. By \cite[Theorem
6]{Tousi-Yassemi}, it follows that for $r\le t$, if $\Delta$ is
$S_r$ but not $S_{r+1}$ and $\Gamma$ is $S_t$ then $\Delta \ast
\Gamma$ is $S_r$ but not $S_{r+1}$.  One may study similar question
for sequentially $S_2$ complexes.

\begin{ques}

Let $\Delta$ and $\Gamma$ be two simplicial complexes. Is it true
that $\Delta \ast\Gamma$ is sequentially $S_2$ if and only if
$\Delta$ and $\Gamma$ are both sequentially $S_2$?

\end{ques}

In particular, it is tempting to show that the join of the
simplicial complexes of two disjoint odd cycles is $S_2$.\\

Munkres \cite[Theorem 3.1]{Mu-84} showed that Cohen-Macaulayness of
a simplicial complex is a topological property. Stanley \cite[Chap.
III, Proposition 2.10]{St-95} proved that sequentially
Cohen-Macaulayness is also a topological property. Recently,
Yanagawa \cite[Theorem 4.5(d)]{Ya-09} proved that Serre's condition
$S_r$ is a topological property as well. Therefore it is natural to
pose the following question.

\begin{ques}

Is sequentially $S_r$ a topological property on simplicial
complexes?

\end{ques}

\end{document}